\documentclass[12pt]{article}
\usepackage{hyperref}
\usepackage{makeidx}
\usepackage{amsfonts}
\usepackage{amsmath}
\usepackage{amssymb}
\usepackage{color,graphicx}
\usepackage{amsopn}
\usepackage{amsthm}
\newtheorem{thm}{Theorem}[section]
\newtheorem{lem}[thm]{Lemma}

\newtheorem{cor}[thm]{Corollary}

\theoremstyle{definition}
\newtheorem{df}[thm]{Definition}
\newtheorem{dfn}[thm]{Definition}

\theoremstyle{remark}

\renewcommand{\t}{\tau}
\renewcommand{\geq}{\geqslant}
\renewcommand{\leq}{\leqslant}
\newcommand{\om}{\mathbb{N}}

\newcommand{\rca}{RCA$_0$ }
\newcommand{\pc}{$\Pi^0_1$ class }

\newcommand{\p}{$\Pi^0_1$ }

\newcommand{\la}{\langle}
\newcommand{\ra}{\rangle}
\newcommand{\nin}{\not\in}
\newcommand{\A}{\forall}
\newcommand{\E}{\exists}

\newcommand{\subs}{\subseteq}
\newcommand{\sups}{\supseteq}

\newcommand{\s}{\sigma}
\newcommand{\emp}{\emptyset}

\renewcommand{\inf}{\infty}
\newcommand{\halts}{\!\downarrow}

\newcommand{\g}{\gamma}
\newcommand{\ext}{\mathrm{Ext}}

\newcommand{\N}{\mathbb{N}}
\newcommand{\br}{\text{Br}}

\renewcommand{\implies}{\rightarrow}

\begin{document}
	\title{Finding paths through narrow and wide trees}

	\author{Stephen Binns\\Bj{\o}rn Kjos-Hanssen \thanks{Partially supported as co-PI by NSF grant DMS-0652669.}}
	\maketitle

	\begin{abstract}
		We consider two axioms of second-order arithmetic. These axioms assert, in two different ways, that \emph{infinite but narrow} binary trees always have infinite paths.
		We show that both axioms are strictly weaker than Weak K\"onig's Lemma, and incomparable in strength to the dual statement (WWKL) that \emph{wide} binary trees have paths. 
	\end{abstract}

	\tableofcontents

	\section{Introduction}

		We investigate here two new subsystems of second-order arithmetic and compare their logical strengths to those of known systems. We are concerned in particular with subsystems that are strictly weaker than Weak K\"onig's Lemma (WKL) and, more specifically, those consisting of axioms that dictate the existence of infinite paths through binary trees. Of course any axiom that implies the existence of at least one path through \emph{every} infinite binary tree is at least as strong as WKL, but we will apply such path-existence axioms to restricted sets of trees. This direction of inquiry is informed by work done by Simpson \cite{simp}, Giusto, Brown \cite{bgs} and others on the axiom  WWKL, Weak Weak K\"onig's Lemma. WWKL states that every tree of positive measure (defined in terms of first-order properties of the tree) has a path. This is shown to be strictly weaker than WKL and significant reversals were established to theorems of analysis. The fact that WWKL is weaker than WKL appeals to the intuition as one feels that it should be easier to find a path through a tree that has {\em many} paths in some sense. Indeed if one chooses {\em left} or {\em right} at random one always has a nonvanishing probability of finding a path through such a tree. It is perhaps slightly paradoxical that an opposite heuristic is also  applicable. That is, if a tree has few paths in some sense, then it is also relatively easy to find a path. If one imagines that one is at an infinitely extendible node $\s$ on some infinite binary tree and having to make a decision of whether to proceed to the left or right, one can wait until it becomes apparent that the tree above $\s0$ or the tree above $\s1$ is finite (we are assuming here that the set of nodes on the tree  is known). If $\s i$ is finite, we take the path through $\s(1-i)$. This strategy will work in allowing us to decide on an infinitely extendible extension of $\s$ if and only if there is there is only one such  extension - that is if $\s$ is not a {\em branching node} of $T$. The intuition behind our weakening of WKL is that it is easier to find paths through trees that have a  small set of branching nodes - again in some particular sense of {\em small}.

		This leaves just the question about the appropriate definition of {\em small}. There are of course some obvious candidates. One is that the set of branching nodes is finite; another is that the set of branching nodes has a maximal (in the sense of extension) element. In both situations the above strategy of waiting until it becomes clear which direction to take will succeed in finding a path if one starts at this maximal branching node (or at the root of the tree if no branching nodes exist). However these notions of {\em smallness} are too strong in the sense that in both cases \rca proves that every tree with a small set of branching nodes has a path. We thus weaken these notions to give us principles that are strictly stronger than RCA$_0$. We do this with the help of {\em bounding witnesses}.

		Suppose $\Phi(n,m)$ is a formula in second order arithmetic, with free variables $n,m$. We say $f$ is a {\em strong bounding witness} for $\Phi$ if
		\[
			\A n \E m\leq f(n) \Phi(n,m).
		\]
		And we say $f$ is a {\em weak bounding witness} for $\Phi$ if
		\[
			\E^\infty n\E m\leq f(n) \Phi(n,m).
		\]
		The revised concepts of {\em smallness} that we use in the paper are as follows:
		\begin{enumerate}
			\item the set of branching nodes $\br(T)$ is {\em small} if there is no weak witness for the predicate
			\[
				\Phi_1(n,\s)\equiv[\s\in\br(T) \wedge \s\geq n];
			\]
			\item the set of branching  nodes is {\em small} if there is no strong witness for the predicate
			\[
				\Phi_2(\s,\t)\equiv[\s\in \ext(T)\implies(\t\in\br(T)\wedge \t\sups \s)],
			\]
			where $\ext(T)$ is the set of nodes on $T$ with infinitely many extensions.
		\end{enumerate}
		For any infinite binary tree $T$, $\A n\E \s\Phi_1(n,\s)$ is true if and only if $T$ has infinitely many branching nodes, and $\A\s\E\t\Phi_2(\s,\t)$ is true if and only if $T$ is perfect. The two axioms that these give rise to are:
		\begin{enumerate}
			\item VSMALL: If $T$ is an infinite binary tree and  there is no weak bounding witness for the predicate $\Phi_1$, then $T$ has a path.
			\item DIM: If $T$ is an infinite binary tree $T$ and there is no strong bounding witness for the predicate $\Phi_2$, then $T$ has a path.
		\end{enumerate}
		In the next section we give slightly different but equivalent versions of these two axioms. DIM is evidently stronger than VSMALL as if a tree has a strong bounding witness for its perfection, the same witness serves as a weak bounding witness for the infinitude of its set of branching nodes. Our main result is that both DIM and VSMALL are independent of WWKL and the DNR axiom.

		The DNR axiom is related to the simplest way of producing a non-computable function - diagonalisation. Take a computable enumeration of the partial computable functions $\la\varphi_e\ra$. Let $f:\N\rightarrow\N$ be such that  $f(e)\neq \varphi_e(e)$ for all $e$. DNR is the principle that asserts the existence of such a function.

		In \cite{AKLS} it is shown  that DNR is a strictly weaker principle than WWKL, but is still non-constructive in the sense that there are models of \rca in which it fails. In fact it is a general rule that any principle which implies the existence of a non-computable set or function is non-constructive in this sense. This is so because there is a model of \rca all of whose sets are computable. We show that our axioms are independent of both DNR and WWKL.

		We follow standard notation. $\om$ is the set of natural numbers, $2^\om$ is the powerset of $\om$ with standard topology and measure. $\varphi_e$ is the $e^{\text{th}}$ partial recursive function, and $\la . , . \ra$ is a fixed  computable bijection between $\om$ and $\om\times \om$.   We will identify subsets of $\om$ with their characteristic functions without comment. Lowercase Greek letters $\s$, $\t$, $\g$ and so on will be used to denote elements of $2^{<\om}$ and upper case Roman letters $X$, $Y$ and $Z$ will usually denote elements of $2^\om$ or, equivalently, subsets of $\N$. The letters $f$, $g$ and $h$ will be used to denote functions from $\om$ to $\om$, that is for elements of $\om^\om$. $\s\subset X$ or $\s\subs\t$  expresses the fact that the infinite binary sequence $X$ or string $\t$ extends $\s$. $\s i$ is the concatenation of  $\s$ and $\la i\ra $.  A {\em tree} will be a subset of $2^{<\om}$, with $\subs$ as its partial order relation.  The root of all trees will be the empty string $\lambda$. A {\em path} through a tree $T$ is the union of an infinite maximal linearly ordered subset of $T$.  The set of paths through $T$ is denoted $[T]$. The cardinality of a set $D$ is denoted $||D||$.

	\section{The axiom systems VSMALL and DIM}

		It is easy to prove that any infinite binary tree with only finitely many branching nodes must have a path. In fact one can prove, just from the axiom system RCA$_0$, that any infinite binary tree with a maximal branching node must have a path (such a path is $\Delta^0_1$-definable).   In this paper we consider infinite binary trees whose sets of branching nodes are small in some sense.  Translating  into the language of second order arithmetic we describe axioms  which state that if an infinite binary tree has only a small set of branching nodes, then the tree must have a path.

		\begin{df}
			If $T$ is a binary tree, then  $s\in T$ is an {\em extendible node of $T$} if the set
			$\{\t\in T: \t\sups\s\}$ is infinite. The set of extendible nodes of $T$ is denoted $\ext(T)$.
			If $P=[T]$, then by $\ext(P)$ we mean $\ext(T)$. $\s\in T$ is a {\em branching node of $T$} if both $\s0$ and $\s1$ are extendible nodes. The set of branching nodes of $T$ is denoted $\br(T)$, and similarly for $\br(P)$ if $P=[T]$.
		\end{df}

		The two axioms we consider are:

		\begin{quote}
			VSMALL:  If $T$ is an infinite binary tree with the property that there is no function $f:\om\rightarrow \om$ such that
			\[
				\E^\inf n \E\s [f(n) \leq|\s|< f(n+1)\wedge \s\in\br(T)]
			\]
			then $T$ has a path.
		\end{quote}

		VSMALL is evidently a strengthening of the statement, provable in RCA$_0$, that every infinite tree with finitely many branching nodes has an infinite path.
		The second axiom deals with perfect trees. A tree is {\em perfect} if every extendible node of $T$ is extended by a branching node.  The second axiom we look at captures the notion of a  witness to the perfection of a tree.

		\begin{quote}
			DIM: If $T$ is an infinite binary tree with the property that  there is no function $f: \om\rightarrow \om$ such that,
			\[
				\A \s\in\ext(T) \E \t\in\br(T) [ \t\sups\s \wedge |\t|\leq f(|\s|)],
			\]
			then $T$ has a path.
		\end{quote}

		If such an $f$ as referred to in DIM existed, it would also serve to contradict VSMALL. Thus DIM+\rca is sufficient to prove VSMALL. We show in the following sections that DIM+\rca is not sufficient to prove DNR, and that WWKL+\rca is not sufficient to prove VSMALL. Consequently, both DIM and VSMALL are independent of both WWKL and DNR. Finally we prove that VSMALL+\rca is not sufficient to prove DIM.

	\section{$\text{ \rca+WWKL}$ $\not\vdash \text{VSMALL}$}

		The following four lemmas are well-known or are easy extensions of the cited results.

		\begin{dfn}
			If $X=\{x_0<x_1<x_2\dots\}\subs\om$, then {\em the principal function of $Y$} $p_Y$ is the function $n\mapsto x_n$.
		\end{dfn}

		\begin{lem}[\cite{cgm} Theorem 1.2] \label{aedom}
			There exists a c.e.~set $A$ such that for almost all $X\in 2^\om$, and for all $Y\leq_T X$, the principal function of $\om\smallsetminus A$ dominates $Y.$
		\end{lem}

		\begin{lem}[\cite{js} Theorem  5.3]\label{sep}
			A \pc is a {\em separating class} if it is of the form
			\[
				\{X\in2^\om: \A n[(n\in A\implies n\in X) \wedge (n\in B\implies n\nin X)]\}
			\]
			for some disjoint c.e.~sets $A$ and $B$.
			If $S\subs 2^\om$ is a \p separating  class with no computable element, then the set $\{X\!\in\! 2^\om: \E Y\!\in\!S\  X\geq_T Y\}$ has measure zero.
		\end{lem}

		\begin{lem}[\cite{kucera85}]\label{1rand}
			The set of 1-random reals has measure 1.
		\end{lem}

		\begin{df}
			$f\in\om^\om$ is {\em of hyperimmune-free degree} if for all $g\in\om^\om$ such that $g\leq_T f$, $g$ is dominated by a computable function.
			If $h\in\om^\om$, then $f$ is {\em of hyperimmune-free degree relative to $h$} if
			for all $g\in\om^\om$ such that $g\leq_T f$, $g$ is dominated by an $h$-computable function.
		\end{df}

		\begin{lem}[\cite{js}]\label{relhif}
			If $f\in\om^\om$ and if $T$ is an infinite binary  $f$-computable tree, then there is a $g\in [T]$ such that $g\oplus f$ is of hyperimmune-free degree relative to $f$.
			If $f$ is of hyperimmune-free degree, then $g\oplus f$ is of hyperimmune-free degree.
		\end{lem}
		\begin{proof}
			This is just the hyperimmune basis theorem of \cite{js} relativised to $f$.
		\end{proof}

		\begin{df}\label{vsmalldef}
			An infinite binary tree $T$ is {\em very small} if the principal function of its set of branching nodes $\br(T)$ dominates every computable function.
			$P=[T]$ is {\em very small} if $T$ is.
		\end{df}
		If $T$ is an infinite binary tree, then define the {\em branching level set of $T$} to be the set $\{|\t|:  \t\in\br(T)\}$. The principal function of this set we call the {\em branching level function of $T$}. It is proved in \cite{binns-thesis} that an infinite computable binary tree is very small if and only if its branching level function dominates every computable function.  We use this  in the following result.

		\begin{thm}\label{wwklnotvsmall}
			$\text{ \rca+WWKL}$ $\not\vdash \text{VSMALL}$
		\end{thm}
		\begin{proof}
			Let $A$ be as in Lemma \ref{aedom}, and let $A_0$ and $A_1$ be infinite c.e.~sets such that $A_0\cup A_1=A$ and $A_0\cap A_1=\emp$ and such that no computable $X\in2^\om$ separates $A_0$ and $A_1$ (such a partition is possible for any c.e.~set - see \cite{ohashi64}). Let $S$ be the separating class of $A_0$ and $A_1$. Then the principal function of $\om\smallsetminus A$ is the branching level function of $S$.

			Now let $R\in 2^\om$ be such that
			\begin{enumerate}
				\item $R$ is 1-random
				\item $R$ does not compute any element of $S$
				\item For all $g\leq_T R$ the principal function of $\om\smallsetminus A$ dominates $g$.
			\end{enumerate}

			Such an $R$ exists as  the classes from Lemmas \ref{aedom} \ref{sep} and \ref{1rand} are of measure 1 and hence their intersection is non-empty.

			Let $R_i$ be the $i^{{\rm th}}$ column of $R$ (that is let $R_i(j)=R(\la i, j\ra)$ for all $j$) and let $X_n=\bigoplus_{i=0}^n R_i$. We claim that the $\omega$-model $\mathfrak{M}$ whose second order part is given by $\{Y: \E n \ Y\leq_T X_n\}$ is a model of WWKL but not of VSMALL.

			To see that it is a model of WWKL, let $T$ be a tree of positive measure in $\mathfrak{M}$. $T$ is then computable from some $X_n$ and therefore, for every $R'\in2^\om$ which is 1-random relative to $X_n$, $R'\oplus X_n$ computes an element of $[T]$ (see \cite{kucera85}). But $R_{n+1}$ is 1-random relative to $X_n$ as every column of a 1-random is random relative to the join of finitely many  other columns.

			Thus $R_{n+1}\oplus X_n$ computes a path through $T$ and hence $\mathfrak{M}$ is a model of WWKL.

			That it is not a model of VSMALL is seen in the fact that every set in   $\mathfrak{M}$ is computable from $R$ and hence cannot compute a path through $S$ (by 2 above). So $S$ has no paths in $\mathfrak{M}$. But $S$ is a nonempty $\Pi_1^0$ class and is thus the set of paths through some infinite computable tree $T_S$ which must be in $\mathfrak{M}$. Furthermore the branching level function of $T_S$ is the set $\om\smallsetminus A$ which dominates everything computable from $R$. Thus any function dominating the branching level function of $T_S$ is not computable from $R$ and hence not in $\mathfrak{M}$. So $T_S$ satisfies the provisions of VSMALL and yet does not have a path. Therefore $\mathfrak{M}$ is not a model of VSMALL.
		\end{proof}

	\section{{\color{red}\rca+DIM $\not\vdash$ DNR}}\label{sectdimnotdnr}
		{\color{red} The published version of this section contains an error. This error and some of its consequences are here indicated in red.}

		We now prove the second half of the main result. 
		Our strategy is to create a model satisfying DIM whose second-order part consists entirely of noncomplex elements of hyperimmune-free degree. Theorem 6 in \cite{KMS} states that a real wtt-computes a DNR function if and only if it is complex. It is well-known that if $X$ is of hyperimmune-free degree  and $X\geq_T Y$, then $X\geq_{\text{wtt}}Y$. Thus if every real in the model we construct is hyperimmune and noncomplex, no element in the model computes a DNR function and hence our model will not satisfy the DNR axiom. We make use of the following definitions and lemmas.

		\begin{dfn}
		A tree $T$ is {\em computably perfect} if there is a strictly increasing computable function $f$ such that for all $n$ and all $\s\in\ext(T)$ of length $f(n)$, there are at least two distinct extensions $\t_1$, $\t_2$ of $\s$ in $\ext(T)$ of length $f(n+1)$.   If $P=[T]$, then we say $P$ is computably perfect if $T$ is.
		\end{dfn}

		\begin{df}  An infinite tree $T$ is {\em diminutive} if no computable tree $T'$ with $[T']\subs [T]$ is computably perfect. If $f\in\om^\om$, then $T$ is {\em $f$-diminutive} if no $f$-computable subtree  is computably perfect. $P=[T]$ is ($f$-)diminutive if $T$ is.
		\end{df}

		For the rest of the paper we will make extensive use of the concept of the {\em Kolmogorov complexity} of a binary string $\s$, denoted $C(\s)$. For an overview of Kolmogorov complexity see \cite{li-vitanyi}. There will be no need in what follows to distinguish between plain complexity and prefix-free complexity.

		\begin{df}
		A real $X\in2^\om$ is {\em complex}  if there is a computable function $f$ such that
		\[
			\A n [ C(X\upharpoonright f(n)) \geq n].
		\]
		If $g\in\om^\om$, then $X$ is {\em $g$-complex} if there is a computable function $f$ such that
		\[
			\A n [ C^g(X\upharpoonright f(n))\geq n],
		\]
		where $C^g(\s)$ denotes the Kolmogorov complexity relative to $g$ - the shortest description of $\s$ on a universal machine using $g$ as an oracle.
		\end{df}

		In \cite{binns-comp} it is shown that a computable tree contains a computably perfect subtree if and only if it contains a complex path. We relativise this in one direction here.

		\begin{lem}\label{maindim} No path through an infinite $f$-diminutive $f$-computable tree is  $f$-complex.
		\end{lem}
		\begin{proof}

		Suppose $T$ is an $f$-computable and $f$-diminutive infinite binary tree, and $A$ is an infinite path through $T$. Let $g(n)$ be any increasing computable function - a putative witness to  the complexity of the $A$.  Let $u\in\om$ and  define a new computable function $h$ by
		\begin{enumerate}
		\item $h(0)=g(0)$
		\item $h(n+1)=g(2h(n)+u).$
		\end{enumerate}
		Eventually we will  use the Recursion Theorem to choose a $u$ that suits our purposes, but until then we treat it as a fixed parameter.

		The set of paths through $T$ is a $\Pi^f_1$ class  $P$, and we can let $P=\bigcap_s P_s$ where each $P_s$  is a clopen subclass of $2^\om$ and $\la P_s \ra_{s\in \om}$ is an $f$-computable sequence.  Consider the  $\Pi^f_1$ class $Q\subs P$ defined as follows. Let $Q_0=2^\om$. If $Q_s$ has been defined, let $S_s\subs2^{<\om}$ be the set
		\[
			\{\s\in\ext(Q_s): \E n\leq s\bigl[ |\s|=h(n)\wedge \E! \t\in\ext(Q_s) [ \t\sups\s \wedge |\t|=h(n+1) ]\bigr]\}.
		\]
		If $S_s^\ast=\{X\in2^\om: \E \s\in S_s  X\supset\s \}$, then define $Q_{s+1}=P_{s+1}\cap Q_s\smallsetminus S_s^\ast$.  $Q$ is then $\bigcap_s Q_s$.

		$Q$ as defined is apparently computably perfect - witnessed by $h$ - and a $\Pi^f_1$ class, and thus the paths through some $f$-computable computably perfect tree. But as $P$ is diminutive, $Q$ must be empty. Thus there is  a stage $s$ such that $A\in S_s^\ast$, and there is an $n$ such that $A\upharpoonright h(n)\in S_s$.

		Now consider a machine $M$ that works as follows. $M$ takes $\s\in2^\om$ as input and tests to see if it is of length $h(n)$ for some $n\leq |\s|$. If it is, it uses $f$ to search for an $t$ such that $\s\in S_t$.  If it finds such an $t$, it outputs the unique extension of $\s$ of length $h(n+1)$ on $Q_t$.  We call this output $\t$, and we have  $C^f(\t) < |\s| + \mathcal{O}(1).$

		On input $A\upharpoonright h(n)$ $M$ will output the unique string on $Q_s$ of length $h(n+1)$ extending $A\upharpoonright h(n)$. That is, it will output $A\upharpoonright h(n+1)$. So
		\[
			C^f\bigl(A\upharpoonright h(n+1)\bigr)< h(n)+ \mathcal{O}(1),
		\]
		or
		\[
			C^f\bigl(A\upharpoonright g(h(n)+u)\bigr)< h(n) + \mathcal{O}(1).
		\]

		The last equation suggests that if we choose $u$ propitiously, we can  ensure that $g$ does not witness the fact that $A$ is  $f$-complex. We do this now.

		The value of the constant on the right-hand side of the equation depends only on the index for the $\Pi^f_1$ class $Q$ (or more precisely on an index for the sequence $\la S_s\ra_{s\in\om}$), and there is a computable function $k$ such that if $e$ is any index for $Q$ , then $C^f(\t)\leq |\s| + k(e)$. Furthermore, there is a computable function $\varphi$ that, given the parameter $u$, will give an  index for $Q$. Thus
		\[
			C^f\bigl(A\upharpoonright g(h(n)+u)\bigr)\leq h(n) + k(\varphi(u)).
		\]

		Using the Recursion Theorem, we fix a value $e$ such that the $\Pi^f_1$ class with index $\varphi(k(e))$ is equal to the $\Pi^f_1$ class with index $e$.
		Now we choose the parameter $u$ to be   equal to $k(e)$. So we have
		\[
			C^f\bigl(A\upharpoonright g(h(n)+ k(e))\bigr) < h(n) +  k(\varphi(k(e))).
		\]
		But as $\varphi(k(e))$ and $e$ index the same $\Pi^f_1$ class,
		\[
			C^f\bigl(A\upharpoonright g(h(n)+k(e))\bigr)< h(n) + k(e),
		\]
		and $A$ is not $f$-complex.

		\end{proof}
		\begin{dfn}
		A  set $A\in2^\om $ is {\em ($f$-)hyperimmune} if its principal function is not dominated by any ($f$-)computable function.
		\end{dfn}

		In   \cite{KMS} it is shown that a set $X$ is not wtt-reducible to a hyperimmune set if and only if  $X$ is complex. We need a limited  relativisation of this theorem here and we present just the direction we require. The proof closely follows \cite{KMS} Theorem 10.

		\begin{lem}[\cite{KMS}]\label{hyperavoidable} 
		If $f$ is of hyperimmune-free degree and $A$ is not wtt-reducible to an $f$-hyperimmune set, then $A$ is $f$-complex.
		\end{lem}
		\begin{proof}
			Suppose $A$ is not wtt-reducible to any $f$-hyperimmune set. We identify, without further comment, the binary string $\s$ with the natural number whose binary expansion is $1\s$. Suppose $C^f$ is calculated using  the universal $f$-oracle machine $U^f$ and choose $d\in\om$ so that $C^f(\s) \leq |\s|+d$ for all $\s$. Let $g(n)$ be the least $\s$ such that $U^f(\s)\halts\sups A\upharpoonright n.$ Now let
			\[
				B:=\{\la n, \s\ra: g(n)=\s \wedge \A k< n[g(k)\neq\s]\},
			\]
			{\color{red}and observe that $A\leq_{\text{wtt}} B$.} [As pointed out by Laurent Bienvenu and Paul Shafer (personal communication, 2012) it seems that the reduction of $A$ to $B$ also requires oracle access to $f$.] (Given $n$ one finds the maximum $\s$ of length at most  $ n+1+d$ such that $\la m,\s\ra\in B$ for some $m\leq n+1$. Then $U^f(\s, n)=A(n)$ and the queries to $B$ are bounded by $(2^{n+d+2}-1)\times (n+1)$.) Therefore $B$ is not $f$-hyperimmune.
			Let $h$ be an $f$-computable function such that there are more than $2^{n+1}$ elements of $B$ in the rectangle
			\[
				\{0,1,\dots h(n)-1\}\times\{0,1,  \dots 2^{h(n)+d}-1\}.
			\]
			As $h$ is $f$-computable and $f$ is of hyperimmune-free degree, there is a computable function $h'$ such that $h'(n)\geq h(n)$ for all $n$. Therefore there are more than $2^{n+1}$ elements of $B$ in the rectangle
			\[
				\{0,1,\dots h'(n)-1\}\times\{0,1,  \dots 2^{h'(n)+d}-1\}.
			\]

			Then $g(h'(n))\geq 2^{n}$ and $C^f(A\upharpoonright h'(n))\geq n$ and so $A$ is $f$-complex.
		\end{proof}
		{\color{red}
			\begin{cor}
			If $f$ is of hyperimmune-free degree, then $A$ is complex if and only if it is $f$-complex.
			\end{cor}
		}
		\begin{proof}  The right to left direction is trivial. If $f$ is hyperimmune-free, then it it easy to see that all hyperimmune sets are $f$-hyperimmune.  Lemma \ref{hyperavoidable} and Theorem 10 in \cite{KMS} completes the proof.

		\end{proof}

		{\color{red}
		\begin{cor}\label{notdnr}
		If $f$ is of hyperimmune-free degree, then no   path of hyperimmune-free degree through an infinite $f$-computable $f$-diminutive tree  computes a DNR function.
		\end{cor}
		}
		[This is false, as pointed out by Laurent Bienvenu and Paul Shafer (personal communication, 2012): take $f$ to be a ML-random of hyperimmune-free degree, and $A=f$; then $A$ is complex but not $f$-complex.]
		\begin{proof}
		By Lemma \ref{maindim}, any path through such a tree is not $f$-complex. Hence by the previous lemma, it is not complex and does not wtt-compute a DNR function. As it is of hyperimmune-free degree, it does not  Turing compute a DNR function.
		\end{proof}

		{\color{red}
		\begin{thm}\label{dimnotdnr}
		$\text{\rca+DIM}$ $\not\vdash \text{DNR}$

		\end{thm}
	}
	\begin{proof}
		To construct a model of DIM that is not a model of DNR, we first define a sequence $\om=X_0, X_1, \dots$ of elements of $2^\om$ such that for all $i>0$
		\begin{enumerate}
		\item $X_i\leq_T X_{i+1}$,

		\item  $X_i$ is of hyperimmune-free degree,

		\item $X_{i+1}$ is a path through  some $X_i$-diminutive $X_i$-computable tree.
		\end{enumerate}

		The second order part of our  model will be $\{Y:\E i Y\leq_T X_i\}$. As usual $\la.,.\ra$ is a computable bijection from $\om\times\om$ onto $\om$ and we can assume that $\la m, n\ra \geq m$ for all $n$ and $m$. Suppose  that $X_i$ has properties 1, 2, and 3 for all $0<i\leq\la m, n\ra$. Now we define $X_{\la m, n\ra+1}$. If $T=\{n\}^{X_m}$ is not an infinite binary $X_{\la m, n\ra}$-diminutive tree, then  let $X_{\la m, n \ra+1}=X_{\la m, n \ra}$. Otherwise, as $X_m\leq_T X_{\la m, n\ra}$, $T$ is also $X_{\la m, n\ra}$-computable and, by Lemma \ref{relhif}, there is a  $Y\in [T]$ such that $Y\oplus X_{\la m, n\ra}$ is of hyperimmune-free degree relative to $X_{\la m, n\ra}$. As $X_{\la m, n\ra}$ is of hyperimmune-free degree  by 2, $Y\oplus X_{\la m, n\ra}$ is also of hyperimmune-free degree by Lemma \ref{relhif}.   Set $X_{\la m, n\ra+1}=Y\oplus X_{\la m, n\ra}$. Properties 1 and 2 are satisfied immediately.

		To prove property 3, first observe that if $X_{\la m,n\ra}=X_{\la m,n\ra+1}$ , then $X_{\la m,n\ra+1}$ is a path through the nonbranching tree $\{X_{\la m,n\ra}\upharpoonright s: s\in \om\}$.

		Otherwise $X_{\la m,n\ra+1}$ is a path through the tree $T'$ consisting of the extendible nodes of
		\[
			\{f\oplus X_{\la m,n\ra}: f\in[T]\},
		\]
		where $[T]$ is  as above. But the branching nodes of $T'$ are uniformly sparser than the branching nodes of $T$ (each branching node of $T'$ is exactly twice the length of the corresponding branching node in $T$), and as $T$ is $X_{\la m,n\ra}$-diminutive, so is $T'$. As $X_{\la m,n\ra+1}$ is a path through $T'$, property 3 is satisfied.

		The model of DIM will be the $\omega$-model with second order part $\{f\in \om^\om: \E i\  f\leq_T X_i\}$.

		{\color{red}Lemma \ref{notdnr}} then gives that no $X_i$  computes a DNR function for any $i$. As every element in the second-order part of the model is computable from some $X_i$, no element in our model computes a DNR function. So our model does not satisfy the DNR axiom.

		\end{proof}

		{\color{red}Theorems \ref{wwklnotvsmall} and \ref{dimnotdnr} together show that VSMALL and DIM  are independent of both DNR and WWKL.}

		Finally we show that DIM and VSMALL are distinct.

	\section{\rca+VSMALL $\not\vdash$ DIM.}

		In this section we produce a model $\mathfrak{M}$ of VSMALL that is not a model of DIM. To construct  the model we use a method very similar to the one used in Section \ref{sectdimnotdnr} but instead of using diminutive trees, we use {\em very small} trees. Recall from Definition \ref{vsmalldef} that a tree $T$ is {\em very small} if the principal function of $\br(T)$ dominates every computable function. Equivalently, if the principal function of $\{|\s|: \s\in\br(T)\}$ dominates every computable function. Furthermore, instead of using the concept of $f$-complexity to distinguish DIM from DNR as in Theorem \ref{dimnotdnr}, now we use the concept of {\em computable traceability} to distinguish VSMALL and DIM.

		\begin{dfn}
		The {\em canonical index} of a finite set  $\{x_1< x_2< \dots<  x_n\}\subs \om$ is
		\[
			\prod_{i=1}^n p_i^{x_i},
		\]
		where $p_i$ is the $i^{\text{th}}$ prime number. We denote by $D_n$ the finite set with canonical index $n$. A {\em ($h$-)computable array} is an infinite sequence of canonically indexed finite sets $\la D_{r(n)}\ra$  with $r$  a ($h$-)computable function.
		\end{dfn}

		\begin{dfn}\label{strong}
		$f \in\om^\om$ is {\em computably traceable} ({\em relative to $h$}) if for all nondecreasing unbounded  computable functions $\varphi$ and for all $g\leq_T f$, there is a ($h$-)computable array $\la D_{r(n)}\ra$ of finite sets such that

		\begin{enumerate}

		\item $\A n \  || D_{r(n)} || \leq \varphi(n)$

		\item $\A n \  g(n)\in D_{r(n)}$

		\end{enumerate}

		\end{dfn}

		\begin{lem}
		If $f$ is computably traceable, and $g$ computably traceable relative to $f$, then $g$ is computably traceable.
		\end{lem}

		\begin{proof}
		Fix any nondecreasing unbounded computable function $\varphi$ and any function $h\leq_T g$.
		As  $g$ is computably traceable relative to $f$, there is an $f$-computable index function $r(n)$ satisfying

		\begin{enumerate}

		\item $\A n \  || D_{r(n)} || \leq \sqrt{\varphi(n)}$

		\item $\A n \  h(n)\in D_{r(n)}.$

		\end{enumerate}
		As $r\leq_T f$ there is a computable index fuction $s$ satisfying

		\begin{enumerate}

		\item[3.] $\A n \  || D_{s(n)} || \leq \sqrt{\varphi(n)}$

		\item[4.] $\A n \  r(n)\in D_{s(n)}.$

		\end{enumerate}

		Now we can define a computable function $t(n)$ so that
		\[
			D_{t(n)}=\bigcup\{ D_m: m\in D_{s(n)} \text{ and } || D_m || \leq \sqrt{\varphi(n)}\}.
		\]

		For all $n$ $||D_{t(n)}|| \leq \sqrt{\varphi(n)}\cdot ||  D_{s(n)} || \leq \varphi(n)$ so $t(n)$ satisfies clause 1 of Definition \ref{strong}. Also, $r(m)\in D_{s(n)}$ and $|| D_{r(m)} ||\leq \sqrt{\varphi(n)}$ so $D_{r(n)}\subs D_{t(n)}$. Therefore $h(n)\in D_{t(n)}$ and the second clause of Definition \ref{strong} is satisfied. $\varphi$ and $h$ were arbitrary so $g$ is computably traceable.
		\end{proof}

		\begin{dfn}
		Let $T$ be a tree and $P=[T]$, the class of paths through $T$. For convenience we introduce the notation
		\[
			T[n]=P[n]:=\{\s\in\ext(T): |\s|=n\}
		\]
		for the set of extendible nodes of $T$ or $P$ of length $n$.
		\end{dfn}

		\begin{lem}\label{hif}
			If $f\in\om^\om$ and $h$ is a hyperimmune-free element of a very small $\Pi_1^f$ class $P$, then for all $h'\leq_T h$ $h'$ is a hyperimmune-free  element of a very small   $\Pi_1^f$ class.
		\end{lem}

		\begin{proof}
		As $h'\leq_T h$ and $h$ is hyperimmune-free, there is a total computable functional $\Phi$ on $2^\om$ such that $\Phi(h)=h'$ (see \cite{odifreddi} Theorem VI.5.5 and Proposition III.3.2). As $h$ is hyperimmune-free, so is $h'$ and $h'\in\Phi[P]$. A straightforward relativisation of  Theorem 4.3.6  in \cite{binns-thesis}, and using the fact that $f$ is of hyperimmune-free degree, gives that $\Phi[P]$ is a  very small  $\Pi_1^f$.
		\end{proof}
		\begin{lem}
		Let $f\in \om^\om$ and $g$ be a hyperimmune-free element of a very small $\Pi_1^f$ class. Then $g$ is computably traceable relative to $f$.
		\end{lem}

		\begin{proof}

		Let $\varphi$ be any nondecreasing unbounded computable function and suppose $h\leq_T g$. By Lemma \ref{hif}, $h$ is a element of some very small $\Pi_1^f$ class $Q$. Let $\la Q_k \ra$ be an $f$-computable sequence of nested clopen sets such that $Q=\bigcap_k Q_k$.

		As $Q$ is  very small there is an $N$ such that  $|| Q[n+1] || \leq \varphi(n)$ for all $n\geq N$. And hence for each $n\geq N$ there is a $k$ such that  $|| Q_k[n+1] || \leq \varphi(n)$. Let $s=s(n)$ be the least such $k$ for each $n$. Now define an $f$-computable function $r(n)$ by
		\[
			D_{r(n)}=\{\s(n): \s\in Q_s[n+1]\}.
		\]
		It is straightforward now to see that clauses 1 and 2  of Definition \ref{strong} are satisfied for $h$ and $r$ and for cofinitely many $n$. We can then finitely adjust $r$ to satisfy 1 and 2 for all $n$.  As $h$ and $\varphi$ were arbitrary, it follows that $g$ is  computably traceable relative to $f$.

		\end{proof}

		The next lemma is based on a theorem by Robinson and Lachlan (\cite{rob67}, \cite{lach68}). The type of  construction is relatively well-known, but for completeness we include a sketch of the proof.

		\begin{lem}[Lachlan, Robinson]\label{lachrob}
		There is a co-c.e.~hyperimmune set $X\subs\om$ and a strictly increasing computable function $f$  such that for all co-c.e.~$Y\subs X$ there exists infinitely many $n\in\om$ such that $||\{m\in Y: m <   f(n)\}||\geq n$. That is $f$ dominates the principal function of  $Y$ infinitely often.
		\end{lem}

		\begin{proof}(Sketch)
			We begin by fixing a computable partition  of $\mathbb{N}$ into  intervals $\la \alpha_i \ra$ such that $||\alpha_i||= i\cdot2^{i+1}$ for each $i$, and such that $\max \alpha_i +1=\min \alpha_{i+1}$.  In the construction we produce a computable double sequence of finite sets $\la \beta_{i,s}\ra$ such that for each $i$ and each stage $s$,
			\begin{enumerate}
				\item $\beta_{i,0}=\alpha_i$,
				\item $\beta_{i,s+1}\subs\beta_{i,s}$,
			\end{enumerate}
			and we let $\beta_i=\lim_s \beta_{i,s}$.  We produce concurrently a double sequence of markers $\la m_{i,s}\ra$ such that for all $i$ and $s$,
			\begin{enumerate}
				\item $m_{i, 0}=i$,
				\item $m_{i, s}<m_{i+1, s}$,
				\item $m_{i, s+1}=m_{j, s}$ for some $j\geq i$.
				\item $\lim_t m_{i,t}$ exists.
			\end{enumerate}
			If $\lim_t m_{i,t}$ is denoted $m_i$, then the  set $X$ required by the theorem will be $\bigcup_{i\in\om} \beta_{m_i}$.

			The values of $m_{i,s}$ are determined to satisfy the requirement that $X$ be hyperimmune - at each stage $s$ if it appears that some approximation to a computable function is going to dominate the principal function of $X$, then the markers are moved to ensure that this fails to occur (for details on such {\em movable marker} arguments  see for example \cite{soare}).

			To determine the value of $\beta_{i,s+1}$ we find the least $j<i$ such that
			\begin{enumerate}
				\item $W_{j,s}\not\sups \beta_{i,s}$
				\item $||W_{j,s}\cap\beta_{i,s}||\geq\frac12||\beta_{i,s}||$
			\end{enumerate}
			and we let $\beta_{i,s+1}=\beta_{i,s}\cap W_{j,s}$ (where $W_{e,s}$ is some standard enumeration of the c.e.~sets). As $W_{j,s}$ is increasing in $s$ and $\beta_{i,s}$ is decreasing in $s$, $\beta_{i,t}\subs W_{j,t}$ for all $t>s$ and so for any $i$, $\beta_{i,s+1}\neq\beta_{i,s}$ at most $i$ times, and for each such $s$ $||\beta_{i,s+1}||\geq\frac12||\beta_{i,s}||$. Therefore, if $\beta_i=\lim_s\beta_{i,s}$, then $||\beta_i||\geq 2i$ for each $i$.

			To see that $X$ now has the required properties, first recall that the movable-marker construction makes it hyperimmune (the reduction of the $\beta_i$ over time cannot conflict with this requirement, and as $\beta_i$ is nonempty for all $i>0$, $X$ is infinite). Suppose now that $Y=\overline{W_e}$  is an infinite co-c.e.~subset of $X$. Therefore for infinitely many $i$, $W_e\not\sups\beta_i$. For each such $i>e$,
			\[
				||W_e\cap\beta_i||<\frac12||\beta_i||,
			\]
			(or else at some stage $\beta_i$ would have been reduced so that $\beta_i\subs W_e$).  So for infinitely many $i$
			\[
				||\overline{W_e}\cap \beta_i||\geq\frac12||\beta_i||\geq i.
			\]
			The function $f(n):=\max\alpha_n$ (which is independent of $e$) then dominates the principal function of $\overline{W_e}$ infinitely often.
		\end{proof}

		\begin{thm}
			\rca+VSMALL $ \not\vdash$ DIM.
		\end{thm}

		\begin{proof}
			We create a model $\mathfrak{M}$ of VSMALL by the same method as in Theorem \ref{dimnotdnr} using very small trees rather than diminutive trees.
			The lemmas above then give that every element of the second-order part of the model is computably traceable.
			We then produce using Lemma \ref{lachrob} a diminutive \pc that has no computably traceable element and hence no element in $\mathfrak{M}$.

			Let $X$ and $f$ be as in Lemma \ref{lachrob}. Let $X_0$ and $X_1$ be a c.e.~partition of $\overline{X}$ and let
			$S=\{Y: Y\sups X_0 \text{ and } Y\cap X_1=\emp\}$ - the separating set of $X_0$ and $X_1$.
			$S$ is a \pc and $X_0$ and $X_1$ can be chosen so that $S$ has no computable element (see \cite{ohashi64} Theorem 1).
			It is straightforward to show that $\{|\s|: \s\in\br(S)\}=X$ and that $S$ is diminutive  as $X$ is hyperimmune.

			Suppose now that $S$ has a computably traceable element $Z$ and consider the $Z$-computable function $n\mapsto Z\upharpoonright f(n)$.
			As $Z$ is computable traceable, there is a computable sequence of canonically indexed sets of binary strings $\la D_n\ra$ such that
			$Z\upharpoonright f(n)\in D_n$ and $||D_n||<n$ for all $n$. Now let
			\[
				S'=S\cap\{Y\in2^\om: \A n  Y\upharpoonright f(n)\in D_n\}.
			\]
			$S'$ is a \pc and contains $Z$ so is non-empty. $S'$ is perfect as $S$ has no computable element.
			Furthermore $W:=\{|\s|: \s\in\br(S')\}$ is a co-c.e. subset of $X$. But $||\{\s\in\ext(S'):  |\s|=f(n)\}||<n$ for all $n,$  so
			$||\{m\in W: m < f(n)\}||<n$, contradicting the assumptions on $X$ and $f$. Thus $S$ is a diminutive \pc with no computably traceable element.
		\end{proof}

\end{document}